\documentclass[12pt,twoside]{article}
\usepackage[centertags]{amsmath}
\usepackage{secdot}
\usepackage{amsfonts}
\usepackage{amssymb}
\usepackage{amsthm}
\usepackage{newlfont}
\usepackage{graphicx}
\usepackage{t1enc}
\usepackage[latin1]{inputenc}
\usepackage[english]{babel}
\usepackage{hyperref}   
\usepackage[T1]{fontenc}
\usepackage{titling}

\usepackage[centertags]{amsmath}
\usepackage{amsmath,amssymb,amsfonts,amsthm,cite}\usepackage{newlfont}
\usepackage{graphicx}\usepackage{sectsty}
\usepackage[active]{srcltx} 
\usepackage{makeidx}
\makeindex
\usepackage{lipsum}
\usepackage{amsthm}
\makeatletter
\usepackage{setspace}

\usepackage{hyperref}

\newlength{\defbaselineskip}
\newcommand{\setlinespacing}[1]%
           {\setlength{\baselineskip}{#1 \defbaselineskip}}

\newcommand{\bq }{\begin{equation}}
\newcommand{\eq }{\end{equation}}
\newcommand{\bbb }{\begin{eqnarray}}
\newcommand{\eee }{\end{eqnarray}}
\newcommand{\bb }{\begin{eqnarray*}}
\newcommand{\ee }{\end{eqnarray*}}
\newcommand{\ed }{\end{document}}

\newtheorem{thm}{Theorem}[section]

\newtheorem{Definition}[thm]{Definition}
\newtheorem{cor}[thm]{Corollary}

\newtheorem{lem}[thm]{Lemma}

\newtheorem{exa}[thm]{Example}

\numberwithin{equation}{section}

\date{2017}

\begin{document}
	\begin{center}
		\large\textbf{ On centrally-extended Jordan endomorophisms in rings} 
	\end{center}
	\centerline {\Large \bf Aziza Gouda and H. Nabiel}
	\vskip .7em
	\centerline {Department of Mathematics, Faculty of Science, Fayoum University, Fayoum, Egypt }
	\centerline {Department of Mathematics, Faculty of Science,
		Al-Azhar University, Nasr City, Cairo, Egypt}
	\centerline{aga07@fayoum.edu.eg and hnabiel@yahoo.com}

	\begin{abstract}
	The aim of this article is to introduce the concept of centrally-extended Jordan endomorphisms and proving that if $R$ is a non-commutative prime ring of characteristic not two, and $G$ is a CE- Jordan epimorphism  such that $[G(x), x] \in Z(R)$ ($[G(x), x^*] \in Z(R)$) for all $x \in R$, then $R$ is an order in a central simple algebra of dimension at most $4$ over its center or there is an element $\lambda$ in the extended of $R$ such that $G(x) = \lambda x$ ($G(x) = \lambda^* x^*$) for all $x \in R$.
	\end{abstract}
	
	\section{Introduction} \label{sect1}
	
	Throughout this paper, $R$ denotes an associative ring with center $Z(R)$,  $Q_{mr}(R)= Q$ denotes the maximal right ring of quotients of $R$  and the center of $Q_{mr}(R)$ is called the extended centroid of $R$ and denoted by $C$. By a ring with involution $*$, we mean a ring equipped with an involution $*$, it is also called $*$-ring. For $x, y \in R$ we have $[x, y] = xy - yx$ and $x \circ y= xy + yx$.
	In 2016, Bell and Daif \cite{Bell 2016} introduced the concept of centrally-extended derivations and centrally-extended endomorphisms on a ring $R$ with center $Z(R)$. A map $D$ on $R$ is called a centrally-extended derivation (CE-derivation) if $D(x + y) - D(x) - D(y)\in Z(R)$ and $D(xy) - D(x)y -xD(y)\in Z(R)$ for all  $x, y\in	R$. A map $T$ of $R$ is said to be a centrally-extended endomorphism (CE-endomorphism) if for each $x,y\in R$, $T(x + y)- T(x) - T(y) \in
	Z(R)$ and $T (xy)- T(x)T(y)\in Z(R)$. They showed that if $R$ is a semiprime ring with no nonzero central ideals, then every	CE-derivation is a derivation and every CE-epimorphism is an	epimorphism.  

	Through the last five years, the authors in [\cite{El-Sayiad 2016}, \cite{Deken 2019},\cite{Deken and soufi 2020}, \cite{ezzat 2021}, \cite{ezzat 2023}, \cite{ref 16} \cite{El Sofi 2022}, \cite{Ageeb 2022}, \cite{Khaled}] introduced the notions of CE-generalized $(\theta,	\phi)$-derivations,  CE- generalized *-derivations, CE-reverse derivations, CE-generalized reverse derivations, CE-higher derivation, CE-higher *derivation, CE-homoderivation, CE- $\alpha$-homoderivation, multiplicative CE-derivation and multiplicative generalized reverse *CE derivation and proved similar results.

	In \cite{ref 14} Bhushan et al. gave the notion of  a CE- Jordan derivation to be a mapping $D$ of $R$ satisfying $D(x + y) - D(x) - D(y) \in Z(R)$ and $D(x\circ y) - D(x)\circ y - x\circ D(y) \in Z(R)$ for all $x, y \in R$. Also, they gave the notion of a CE- Jordan *-derivation to be a mapping $D$ of a ring $R$ with involution $*$ satisfying $D(x + y) - D(x) - D(y) \in Z(R)$ and $D(x\circ y) - D(x)y^* - xD(y) - D(y)x^* - yD(x) \in Z(R)$ for all $x, y \in R$. They proved the following. Let $R$ be a  non-commutative prime ring(with involution $*$) of characteristic not 2. If $R$ admits a CE-Jordan derivation $d$ of $R$ such that $[d(x), x] \in Z(R)$ ($[d(x), x^*] \in Z(R)$) for all $x \in R$, then either $d = 0$ or $R$ is an order in a central simple algebra of dimension at most $4$ over its center. Also they obtained the same results in the case of CE- Jordan *-derivations.
	
	 Bhushan et al. in \cite{ref 15} introduced the notions of  CE- generalized derivation  Jordan derivation is a mapping $F$ of $R$ satiesfying $F(x + y) - F(x) - F(y) \in Z(R)$ and $F(x\circ y) - D(y)x - yD(x)\in Z(R)$ for all $x, y \in R$, where $D$ is a CE-Jordan derivation of $R$. They proved the following: (i) Let $R$ be a  non-commutative prime ring(with involution $*$) of characteristic not 2. If $R$ admits a CE-generalized derivation $F$ constrained with a CE-Jordan derivation $d$ of $R$ such that $[F(x), x] \in Z(R)$ ($[d(x), x^*] \in Z(R)$) for all $x \in R$, then $R$ is an order in a central simple algebra of dimension at most $4$ over its center or $F(x) = \lambda x$ for all $x \in R$, where is $\lambda \in C$. (ii) Let $R$ be a  non-commutative prime ring with involution $*$ and characteristic not 2. If $R$ admits a CE-generalized derivation $F$ constrained with a CE-Jordan derivation $d$ of $R$ such that $[F(x), x^*] \in Z(R)$) for all $x \in R$, then $R$ is an order in a central simple algebra of dimension at most $4$ over its center or $F = 0$.

	Our Objective in this article is to introduce the notion of CE- Jordan endomorphisms and prove the following results:
	
	(i) Let $R$ be a non-commutative prime ring of characteristic not two. If $R$ admits a CE- Jordan epimorphism $G$ such that $[G(x), x] \in Z(R)$ for all $x \in R$, then $R$ is an order in a central simple algebra of dimension at most $4$ over its center or $G(x) = \lambda x$ for all $x \in R$, where is $\lambda \in C$.
	
	(ii) Let $R$ be a  non-commutative prime ring of characteristic not two with involution $*$. If $R$ admits a CE- Jordan epimorphism $G$ such that $[G(x), x^*] \in Z(R)$ for all $x \in R$, then $R$ is an order in a central simple algebra of dimension at most $4$ over its center or $G(x) = \lambda^*x^*$ for all $x \in R$, where is $\lambda \in C$.	

\section{Preliminary results} \label{sect2}
	
The following lemmas help us to prove our results.

\begin{lem} \cite[Theorem 3.2]{ref 12} \label{lem2.2} Let $R$ be a prime ring. If  $g: R \longrightarrow R$ is an additive mapping such that $[g(x), x] = 0$ for all $x \in R$, then there exists $\lambda \in C$ and an additive mapping $\beta: R \longrightarrow C$, such that $g(x) = \lambda x + \beta(x)$ for all $x \in R$.
\end{lem}
\begin{lem} \cite[Lemma 1]{ref 13} \label{lem2.3} Let $R$ be a prime ring with extended centroid $C$. Then the following statements are equivalent:
	(i) $R$ satisfies $s_{4}$.
	(ii) $R$ is commutative or $R$ embeds into $M_{2} (K)$, for a field $K$.
	(iii) $R$ is algebraic of bounded degree $2$ over $C$.
	(iv) $R$ satisfies $[[x^2, y], [x, y]]$.
\end{lem}
In the following lemma, we generalize \cite[Proposition 3.1]{ref 12} to use it in our results.
\begin{lem}\label{lem3.3}
	Let $R$ be a $2$-torsion free semiprime ring with a nonzero map $\alpha$: $R \longrightarrow R$ and $\alpha(x + y) - \alpha(x) - \alpha(y) \in Z(R)$. If $[\alpha(x), x] \in Z(R)$ for all $ x \in R$, then $[\alpha(x), x] = 0$ for all $x \in R$.
\end{lem}
\begin{proof}
	By our hypothesis, we have
	\begin{equation}\label{eq35}
		[\alpha(x), x] \in Z(R)  \: \textit{ for all } \: x \in R. 
	\end{equation}
	Linearizing (\ref{eq35}) implies that
	\begin{equation}\label{eq36}
		[\alpha(x), y] + [\alpha(y), x] \in Z(R)  \: \textit{ for all } \: x, y \in R. 
	\end{equation}
	Putting $x^2$ instead of $y$ in (\ref{eq36}) and using (\ref{eq35}) yield  
	\begin{equation}\label{eq37}
		2[\alpha(x), x]x + [\alpha(x^2), x] \in Z(R)  \: \textit{ for all } \: x \in R. 
	\end{equation}
	From (\ref{eq35}) and (\ref{eq37}), we obtain
	\begin{equation*}
		\begin{split}
			0 = &[\alpha(x), 2[\alpha(x), x]x + [\alpha(x^2), x]] = 2[\alpha(x), x]^2 + [\alpha(x), [\alpha(x^2), x]] \: \textit{ for all } \: x \in R. 
		\end{split}
	\end{equation*}
	From (\ref{eq35}), we have $[\alpha(x^2), x^2] = [\alpha(x^2), x]x + x[\alpha(x^2), x]\in Z(R)$ for all $x\in Z(R)$
	Thus,
	\begin{equation} \label{eq38}
		\begin{split}
			0 = &[\alpha(x), [\alpha(x^2), x]x + x[\alpha(x^2), x]] = 2[\alpha(x), x][\alpha(x^2), x] + [\alpha(x), [\alpha(x^2), x]]x \\&+ x[\alpha(x), [\alpha(x^2), x]] = 2[\alpha(x), x][\alpha(x^2), x] - 4x[\alpha(x), x]^2\: \textit{ for all } \: x \in R. 
		\end{split}
	\end{equation}
	So, we find
	\begin{equation} \label{eq40}
		0 = [\alpha(x),  2[\alpha(x), x][\alpha(x^2), x] - 4x[\alpha(x), x]^2] = - 8[\alpha(x), x]^3 \: \textit{ for all } \: x \in R. 
	\end{equation}	
	The two torsion freeness of $R$ and the semiprimeness imply that $[\alpha(x), x] = 0$  for all  $x \in R$. 	
\end{proof}
By the following lemma, we say \cite[Proposition 2.2]{ref 11} is true when we take $\alpha$: $R \longrightarrow R$ such that $\alpha(x + y) - \alpha(x) - \alpha(y) \in Z(R)$ for all $ x, y \in R$ instead of $\alpha$ is additive. Also, we use it in proving of our results.
\begin{lem}\label{lem3.4}
	Let $R$ be a $2$-torsion free semiprime ring admitting an involution $*$. If $\alpha$: $R \longrightarrow R$ a nonzero map, $\alpha(x + y) - \alpha(x) - \alpha(y) \in Z(R)$ for all $ x, y \in R$ and $[\alpha(x), x^*] \in Z(R)$ for all $x \in R$, then $[\alpha(x), x^*] = 0$ for all $x \in R$.
\end{lem}
\begin{proof}	
	By our assumpation, we have $[\alpha(x), x^*] \in Z(R)$ for all $ x \in R$. Putting $x^*$ instead of $x$, we get $[\alpha(x^*), x] \in Z(R)$ for all $ x \in R$. Define a map $\gamma$: $R \longrightarrow R$ such that $\gamma(x) = \alpha(x^*)$ for all $x \in R$. We note $\gamma(x + y) - \gamma(x) - \gamma(y)\in Z(R)$ for all $x \in R$ and $[\gamma(x), x] \in Z(R)$ for all $x \in R$. Lemma (\ref{lem3.3}) gives $[\gamma(x), x] = 0$ for all $x \in R$. This implies that $[\alpha(x^*), x]= 0$ for all $x \in R$, hence $[\alpha(x), x^*] = 0$ for all $x \in R$.
\end{proof}

\section{Centrally-extended  Jordan endomorphism} \label{sect3}
\begin{Definition}
 A map $G$: $R \longrightarrow R$ is called a centrally-extended  Jordan endomorphism (CE-Jordan endomorphism) if $G(x + y) - G(x) - G(y) \in Z(R)$ and $G(x \circ y) - G(x) \circ G(y) \in Z(R)$ for all $ x, y \in R$.\\
\end{Definition}

In the following examples, we make sure the existance of CE-Jordan endomorphisms.	
\begin{exa}\label{exa3.1}
	Suppose R = $M_{2} (\mathbb{Z})$ where $\mathbb{Z}$ is the ring of integers.   It is obvious that the mapping $G$: $R \longrightarrow R$ defined by $G$ $(\begin{matrix}
		a & b \, \\
		c & t \, \\
	\end{matrix}
	)$ = $(\begin{matrix}
		c & 0 \, \\
		0 & c \, \\
	\end{matrix})$ is a CE-Jordan endomorphism but not a Jordan endomorphism.
\end{exa}

\begin{exa}\label{exa3.2}
 Let R = $M_{2} (\mathbb{Z})$. The mapping $G$: $R \longrightarrow R$ defined by $G$$(\begin{matrix}
	a & b \, \\
	c & t \, \\
\end{matrix}
)$ = $(\begin{matrix}
	t & -b \, \\
	-c & a \, \\
\end{matrix})$ is a CE-Jordan endomorphism but not a CE-endomorphism.
\end{exa}




Example \ref{exa3.1} shows that Bell and Daif \cite[, Theorem 2.8]{Bell 2016} can not be extended to the case of CE-endomorphisms but fortunately, we proved the additive condition in the following lemma.

\begin{lem}\label{lem3.1}
Let $R$ be a $2$-torsion free ring with no nonzero central ideals. Then every  CE- Jordan epimorphism $G$ of $R$ is additive.
\end{lem}
\begin{proof}
Since $G$ is a CE-endomorphism, then we have for $x, y, w \in R$
\begin{equation}\label{eq3.1}
		G(x + y) = G(x) + G(y) + a_{1},   \:   a_{1} \in Z(R).              
\end{equation}
And,
	\begin{equation}\label{eq3.2}
		G(w\circ(x + y)) = G(w)G(x + y) + G(x + y)G(w) + a_{2},    \:     a_{2} \in Z(R).            
	\end{equation}
	By (\ref{eq3.1}), we get 
	\begin{equation}\label{eq3.3}
		G(w\circ(x + y)) = G(w)(G(x) + G(y) + a_{1}) + (G(x) + G(y) + a_{1})G(w) + a_{2}.                    
	\end{equation}
	Also, we have
	\begin{equation}\label{eq3.4}
		\begin{split}   
			G(w\circ(x + y)) &= G((w\circ x) + (w\circ y)) = G(w\circ x) + G(w\circ y) + a_{3} \\&= G(w)G(x) + G(x)G(w) + a_{4} + G(w)G(y) + G(y)G(w)\\& + a_{5} + a_{3},   \:      a_{3}, a_{4}, a_{5} \in Z(R). 
		\end{split}                     
	\end{equation}
	Equation (\ref{eq3.3}) and (\ref{eq3.4}) imply that $2a_{1}G(w) + a_{2} = a_{3} + a_{4} + a_{5}$, i.e., $2a_{1}G(w) \in Z(R)$ for all $w \in R$.
	Since $R$ is $2$-torsion free, then $a_{1}G(w) \in Z(R)$ for all $w \in R$. Using $G$ is surjective gives $aw \in Z(R)$ for all $w \in R$. 	
	Thus, $a_{1}R$ is a central ideal, hence $a_{1}R = (0)$. Therefore, letting $A(R)$ be the two-sided annihilator of $R$, we have
	$a_{1} \in A(R)$. But $A(R)$ is a central ideal, so $a_{1} = 0$ and by (\ref{eq3.1}), we get our proof.
\end{proof}

\begin{cor}\label{cor3.4}
	Let $R$ be a non-commutative prime ring of characteristic not two. If $G$ is a CE-Jordan epimorphism  of $R$, then $G$ is additive.
\end{cor}

Now, we are ready to prove our first theorem.
\begin{thm}\label{thm 3.5}
	Let $R$ be a non-commutative prime ring of characteristic not two. If $R$ admits a CE- Jordan epimorphism $G$ such that $[G(x), x] \in Z(R)$ for all $x \in R$, then $R$ is an order in a central simple algebra of dimension at most $4$ over its center or $G(x) = \lambda x$ for all $x \in R$, where is $\lambda \in C$.
\end{thm}
\begin{proof}
	 By our hypothsis and Lemma (\ref{lem3.3}), we obtain
	\begin{equation}\label{eq3.5}
		[G(x), x] = 0  \: \textit{ for all } x \in R.             
	\end{equation}
	Linearizing (\ref{eq3.5}) gives
	\begin{equation}\label{eq3.6}
		[G(x), y] + [G(y), x] = 0  \: \textit{ for all } x, y \in R.            
	\end{equation}
	Substituting $x \circ y$ for $y$ in (\ref{eq3.6}) gives 
	\begin{equation*}
		[G(x), x \circ y] + [G(x \circ y), x] = 0  \: \textit{ for all } x, y \in R.                   
	\end{equation*}
	Thus,
	\begin{equation}\label{eq3.7}
		\begin{split}   
			0& = [G(x), x \circ y] + [G(x)G(y) + G(y)G(x), x]\\& = [G(x), x \circ y] + G(x)[G(y), x] + [G(y), x]G(x)   \: \textit{ for all } x, y \in R.                   
		\end{split}                     
	\end{equation}
	Corollary (\ref{cor3.4}) implies that $G$ is additive, then (\ref{eq3.5}) and Lemma (\ref{lem2.2}) give $G(x) = \lambda x + \beta(x)$ for all $x \in R$ where $\beta: R \longrightarrow C$ is an additive mapping. By (\ref{eq3.7}), we arrive at	  
	\begin{equation}\label{eq3.8}
		\begin{split}   
			0& = [\lambda x + \beta(x), x \circ y] + (\lambda x + \beta(x))[\lambda y + \beta(y), x] + [\lambda y + \beta(y), x](\lambda x + \beta(x))\\&  \: \textit{ for all } x, y \in R.                   
		\end{split} 	
	\end{equation}	
	Since $\lambda$ and  $\beta(x) \in C$ for all $x \in R$, then
	\begin{equation}\label{eq3.9}
		\begin{split}   
			0& = \lambda[x, x \circ y] + \lambda^{2} x[y, x]  + \lambda^{2}[y, x]x + 2\lambda \beta(x)[y, x] \\& = \lambda[x^{2}, y] + \lambda^{2} [y, x^{2}] + 2\lambda \beta(x)[y, x]\\& = (\lambda^{2} - \lambda) [y, x^{2}] + 2\lambda \beta(x)[y, x]\: \textit{ for all } x, y \in R.                   
		\end{split} 	
	\end{equation}	
	So, we find
	\begin{equation*}
		0 = (\lambda^{2} - \lambda) [[x^{2}, y], [y, x]] \: \textit{ for all } x, y \in R.                   	
	\end{equation*}	
	The primeness of $R$ leads to the primeness of $Q$, then we get either $\lambda^{2} = \lambda$ or $[[x^{2}, y], [x, y]]  = 0$ for all  $x, y \in R$. By Lemma (\ref{lem2.3}), the second case is equivalent to the $s_{4}$ identity and $R$ is assumed to be non-commutative, therefore $R$ is an order in a central simple algebra of dimension at most $4$ over $Z(R)$. \\
	On the other hand, let us assume that $\lambda^{2} = \lambda$. Then from (\ref{eq3.9}), we have $2\lambda \beta(x)[y, x] = 0$ for all  $x, y \in R$. Since characteristic $R$ is not two, then characteristic $Q$ is not two and so, we find $\lambda \beta(x)[y, x] = 0$ for all  $x, y \in R$. The primeness of $Q$ gives $\lambda = 0$ or $\beta(x)[y, x] = 0$ for all  $x, y \in R$. 
	Assume $\lambda = 0$. Then $G(x) = \beta(x) \in C$ for all  $x \in R$. Since $G$ is surjective, then $R \subseteq C$, i.e., $R$ is a commutative ring, a contradiction. If $\beta(x)[y, x] = 0$ for all  $x, y \in R$, then the primeness of $Q$ implies that for each $x \in R$ either $\beta(x) = 0$ or $[y, x] = 0$ for all  $y \in R$. Write  $E_{1}  =\{x \in R: \beta(x) = 0\}$ and  $E_{2}  =\{x \in R: [y, x] = 0 \: \textit{ for all } y \in R\}$. Therefore, we note that $E_{1} $ and $E_{2}$ are  proper additive subgroups of $R$ and it is union equls $R$ which is impossible. Thus, either $R = E_{1} $ or $R = E_{2} $. It implies that either $\beta(x)= 0$ for all  $x \in R$ or $[y, x] = 0$ for all  $x, y \in R$.
	If $[y, x] = 0$ for all  $x, y \in R$, then $R$ is commutative, a contradiction. If  $\beta(x)= 0$ for all  $x \in R$, then  $G(x) = \lambda x$ for all  $x \in R$. 
\end{proof}

Now, we prove the second theorem that dealing with ring with involution.
\begin{thm}\label{thm 3.3}
	Let $R$ be a  non-commutative prime ring of characteristic not two with involution $*$. If $R$ admits a CE- Jordan epimorphism $G$ such that $[G(x), x^*] \in Z(R)$ for all $x \in R$, then $R$ is an order in a central simple algebra of dimension at most $4$ over its center or $G(x) = \lambda^*x^*$ for all $x \in R$, where is $ \lambda \in C$.
\end{thm}
\begin{proof}
Lemma (\ref{lem3.4}) gives
	\begin{equation}\label{eq3.10}
		[G(x), x^*] = 0  \: \textit{ for all } x \in R.             
	\end{equation}
	Applying involution in (\ref{eq3.10}), we get	
	\begin{equation}\label{eq3.11}
		[G(x)^*, x] = 0  \: \textit{ for all } x \in R.             
	\end{equation} 
	Linearizing (\ref{eq3.11}) gives
	\begin{equation}\label{eq3.12}
		[G(x)^*, y] + [G(y)^*, x] = 0  \: \textit{ for all } x, y \in R.            
	\end{equation}
	Substituting $x \circ y$ for $y$ in (\ref{eq3.12}) gives 
	\begin{equation*}
		[G(x)^*, x \circ y] + [G(x \circ y)^*, x] = 0  \: \textit{ for all } x, y \in R.                   
	\end{equation*}
	Thus,
	\begin{equation}\label{eq3.13}
		\begin{split}   
			0& = [G(x)^*, x \circ y] + [(G(x)G(y) + G(y)G(x))^*, x]\\& = [G(x)^*, x \circ y] + G(x)^*[G(y)^*, x] + [G(y)^*, x]G(x)^*   \: \textit{ for all } x, y \in R.                   
		\end{split}                     
	\end{equation}
	Since $*G$ is additive, then by (\ref{eq3.11}) and Lemma (\ref{lem2.2}), we find $G(x)^* = \lambda x + \beta(x)$ for all $x \in R$ where $\beta: R \longrightarrow C$ is an additive mapping. So, (\ref{eq3.13}) becomes
	\begin{equation}\label{eq3.14}
		\begin{split}   
			0& = [\lambda x + \beta(x), x \circ y] + (\lambda x + \beta(x))[\lambda y + \beta(y), x] + [\lambda y + \beta(y), x](\lambda x + \beta(x))\\&  \: \textit{ for all } x, y \in R.                   
		\end{split} 	
	\end{equation}	
We have $\lambda$ and  $\beta(x)$ for all $x \in R$ belong to $C$, then		
\begin{equation}\label{eq3.15}
		(\lambda^{2} - \lambda) [y, x^{2}] + 2\lambda \beta(x)[y, x] = 0\: \textit{ for all } x, y \in R.                                   	
\end{equation}	
So, we get $ (\lambda^{2} - \lambda) [[x^{2}, y], [y, x]] = 0$ for all  $x, y \in R$. As above, we find  either $\lambda^{2} = \lambda$ or $R$ is an order in a central simple algebra of dimension at most $4$ over its center. If $\lambda^{2} = \lambda$, then eqution (\ref{eq3.15}) implies that $\beta(x)[y, x] = 0$ for all $x, y \in R$. By the same manner as in Theorem (\ref{thm 3.5}) we arrive at  $\beta(x) = 0$ for all $x \in R$. Therefore $G(x)^* =\lambda x$ for all $x \in R$. The involution leads to $G(x) =\lambda^* x^*$ for all $x \in R$, where $ \lambda \in C$. 
\end{proof}

\end{document}